\documentclass[a4paper,11pt]{article}
\usepackage[margin=1in]{geometry}
\usepackage{tikz}
\usepackage[small,bf,hang]{caption}
\usepackage{siunitx}
\usepackage{amsmath,amsthm,amssymb}
\usepackage{graphicx}
\usepackage{subfig}
\usepackage{tikz}
\usepackage{float}
\usepackage{hyperref}
\hypersetup{colorlinks=true,    citecolor = blue,    linkcolor=black,    filecolor=magenta,    urlcolor=cyan}
  
\newtheorem{theorem}{Theorem}

\newtheorem{lemma}{Lemma}

\newtheorem{conjecture}{Conjecture}
\title{Recoloring via modular decomposition}

\author{Manoj Belavadi
 \thanks{Department of Mathematics, Wilfrid Laurier University,
 Waterloo, ON, Canada, N2L 3C5. Email: \texttt{mbelavadi@wlu.ca}. ORCID: 0000-0002-3153-2339. Research supported by the Natural Sciences and Engineering Research Council of Canada (NSERC) grant RGPIN-2016-06517.}
 \and Kathie Cameron
 \thanks{Department of Mathematics, Wilfrid Laurier University,
 Waterloo, ON, Canada, N2L 3C5. Email: \texttt{kcameron@wlu.ca}. ORCID: 0000-0002-0112-2494. Research supported by the Natural Sciences and Engineering Research Council of Canada (NSERC) grant RGPIN-2016-06517.}
 \and Ni Luh Dewi Sintiari
 \thanks{Department of Informatics, Universitas Pendidikan Ganesha,
 Bali, Indonesia, 81116. Email: \texttt{nld.sintiari@gmail.com}. ORCID: 0000-0002-6562-4434. Research supported by the Natural Sciences and Engineering Research Council of Canada (NSERC) grant RGPIN-2016-06517.}}

\begin{document}

\maketitle

\begin{abstract}
    The reconfiguration graph of the $k$-colorings of a graph $G$, denoted $R_{k}(G)$, is the graph whose vertices are the $k$-colorings of $G$ and two colorings are adjacent in $R_{k}(G)$ if they differ in color on exactly one vertex. A graph $G$ is said to be recolorable if $R_{\ell}(G)$ is connected for all $\ell \geq \chi(G)$+1. We demonstrate how to use the modular decomposition of a graph class to prove that the graphs in the class are recolorable. In particular, we prove that every ($P_5$, diamond)-free graph, every ($P_5$, house, bull)-free graph, and every ($P_5$, $C_5$, co-fork)-free graph is recolorable.

     A graph is prime if it cannot be decomposed by modular decomposition except into single vertices. For a prime graph $H$, we study the complexity of deciding if $H$ is $k$-colorable and the complexity of deciding if there exists a path between two given $k$-colorings in $R_{k}(H)$. Suppose $\mathcal{G}$ is a hereditary class of graphs. We prove that if every blowup of every prime graph in $\mathcal{G}$ is recolorable, then every graph in $\mathcal{G}$ is recolorable.
\end{abstract}

\section{Introduction}

Let $G$ be a finite simple graph with vertex-set $V(G)$ and edge-set $E(G)$. We assume that $G$ is connected unless stated otherwise. We use $n$ to represent the number of vertices in a graph. Two vertices $u$ and $v$ are \emph{adjacent} in $G$ if $uv\in E(G)$. For a positive integer $k$, a \emph{$k$-coloring} of $G$ is a mapping from $V(G)$ to a set of colors $\{1,2,\dots,k\}$ such that no two adjacent vertices receive the same color. We say that $G$ is \emph{$k$-colorable} if it admits a $k$-coloring, and the \emph{chromatic number} of $G$, denoted $\chi(G)$, is the minimum number of colors required to color $G$. We say $\chi$-coloring of $G$ instead of $\chi(G)$-coloring of $G$, when appropriate. The \emph{reconfiguration graph of the $k$-colorings}, denoted $R_k(G)$, is the graph whose vertices are the $k$-colorings of $G$ and two colorings are adjacent in $R_k(G)$ if they differ in color on exactly one vertex. A graph $G$ is said to be \emph{$k$-mixing} if $R_{k}(G)$ is connected. Given a reconfiguration graph, we can ask the following questions: Is the reconfiguration graph connected? If so, what is the diameter of the reconfiguration graph? In this paper we focus on the connectivity of the reconfiguration graph of the $k$-colorings. A graph is said to be \emph{$k$-mixing} if $R_{k}(G)$ is connected. We say a graph $G$ is \emph{recolorable} if $G$ is $\ell$-mixing for all $\ell\geq \chi(G)$+1. 

Deciding whether there exists a path between two colorings in $R_k(G)$ is PSPACE-complete for $k > 3$ \cite{Bonsma2009}, and can be decided in polynomial time for $k\leq 3$ \cite{Cereceda2008}. In 2018, Wrochna proved that the problem remains PSPACE-complete for graph classes with bounded bandwidth, and hence for graph classes with bounded treewidth \cite{Wrochna}. This led researchers to study the problem for various restricted graph classes. In \cite{bonamy2014}, Bonamy et. al. introduced the class of $k$-color-dense graphs and proved that every $k$-color-dense graph is recolorable. The class of $k$-color-dense graphs includes the class of $k$-colorable chordal graphs. For a graph $H$, a graph $G$ is $H$-$free$ if no induced subgraph of $G$ is isomorphic to $H$. In \cite{Belavadi}, it was proved that every $H$-free graph $G$ is recolorable if and only if $H$ is an induced subgraph of $P_4$ or $P_3$+$P_1$. For a collection of graphs $\mathcal{H}$, $G$ is $\mathcal{H}$-free if $G$ is $H$-free for every $H \in \mathcal{H}$. For more results about reconfiguration of vertex colorings in forbidden induced subgraph classes see \cite{Belavadi2024}.

Bonamy and Bousquet \cite{Bonamy} used the tree decomposition of graphs to show that every graph $G$ is ($tw(G)$+2)-mixing, where $tw(G)$ denotes the treewidth of the graph $G$. In this paper we demonstrate how to use the modular decomposition of a graph $G$ to prove that $G$ is $\ell$-mixing for $\ell\ge \chi(G)$+1. 

A non-empty set $S\subseteq V(G)$ is a \emph{module} if every vertex in $V(G)\setminus S$ is either adjacent to every vertex of $S$ or no vertex of $S$. A module $S\subset V(G)$ is said to be \emph{non-trivial} if $S$ contains at least two vertices. Note that for each $v\in V(G)$, $\{v\}$ is a trivial module in $G$. A graph $G$ is \emph{prime} if it does not contain any non-trivial module.

By \emph{substituting a graph} $H$ for $S\subseteq V(G)$ we mean taking the graph $G$-$S$ and adding an edge between every vertex of $H$ and every vertex of $G$-$S$ that is adjacent to some vertex of $S$ in $G$. Note that this is slightly different from the well-known definition of substituting a graph for a vertex. However, substituting a graph $H$ for $S\subseteq V(G)$ is equivalent to contracting the set $S$ to a vertex $v$ (and removing any loops and multiple edges created), and then applying the usual definition of substituting the graph $H$ for the vertex $v$. A \emph{blowup} of a graph $G$ is a graph obtained by substituting non-empty cliques for some vertices of $G$.

\begin{figure}[h]
\centering
\subfloat[$diamond$]{
\begin{tikzpicture}[scale=0.75]
\tikzstyle{vertex}=[circle, draw, fill=black, inner sep=0pt, minimum size=4pt]
\node[vertex] (1) at (0,0) {};
\node[vertex] (2) at (-1,1) {};
\node[vertex] (3) at (1,1) {};
\node[vertex] (4) at (0,2) {};
\draw (1)--(2); \draw (1)--(3); \draw (2)--(3); \draw (4)--(2); \draw (4)--(3);
\end{tikzpicture}
} \hspace{1.5cm}
\subfloat[$house$]{
\begin{tikzpicture}[scale=0.75]
\tikzstyle{vertex}=[circle, draw, fill=black, inner sep=0pt, minimum size=4pt]
\node[vertex] (1) at (0,0) {};
\node[vertex] (2) at (0,1) {};
\node[vertex] (3) at (1,1) {};
\node[vertex] (4) at (1,0) {};
\node[vertex] (5) at (0.5,2) {};
\draw (2)--(1); \draw (1)--(4); \draw (3)--(4); \draw (2)--(3); \draw (5)--(3); \draw (2)--(5);
\end{tikzpicture}
} \hspace{1.5cm}
\subfloat[$bull$]{
\begin{tikzpicture}[scale=0.75]
\tikzstyle{vertex}=[circle, draw, fill=black, inner sep=0pt, minimum size=4pt]
\node[vertex] (1) at (0,0) {};
\node[vertex] (2) at (0,1) {};
\node[vertex] (3) at (1,1) {};
\node[vertex] (4) at (1,0) {};
\node[vertex] (5) at (0.5,-1) {};
\draw (2)--(1); \draw (1)--(4); \draw (3)--(4); \draw (5)--(1); \draw (4)--(5);
\end{tikzpicture}
} \hspace{1.5cm}
\subfloat[$co$-$fork$]{
\begin{tikzpicture}[scale=0.75]
\tikzstyle{vertex}=[circle, draw, fill=black, inner sep=0pt, minimum size=4pt]
\node[vertex] (1) at (0,0) {};
\node[vertex] (2) at (1,0.5) {};
\node[vertex] (3) at (1,-0.5) {};
\node[vertex] (4) at (2,0) {};
\node[vertex] (5) at (3,0) {};
\draw (2)--(1); \draw (1)--(3); \draw (3)--(4); \draw (2)--(3); \draw (2)--(4); \draw (4)--(5);
\end{tikzpicture}
}
\caption{}\label{fig:u5}
\end{figure}

A class of graphs $\mathcal{G}$ is said to be \emph{hereditary} if for every $G$ in $\mathcal{G}$, every induced subgraph of $G$ is also in $\mathcal{G}$. Note that for a collection of graphs $\mathcal{H}$, the class of $\mathcal{H}$-free graphs is hereditary. In Section \ref{sec:comp}, for a prime graph $H$, we study the complexity of deciding if $H$ is $k$-colorable and the complexity of deciding if there exists a path between two given $k$-colorings in $R_{k}(H)$.

\begin{theorem}\label{thm:primecoloring}
    Given a prime graph $H$, for all $k\ge 3$, deciding whether $H$ is $k$-colorable is NP-complete.
\end{theorem}

\begin{theorem}\label{thm:primerecoloring}
    Given a prime graph $H$, for all $k\ge 4$, deciding whether there exists a path between two given colorings in $R_{k}(H)$ is PSPACE-complete.
\end{theorem}

In Section \ref{modulardecomposition}, we will demonstrate how to use modular decomposition in reconfiguration of vertex coloring.

\begin{theorem}\label{hereditary}
    Let $\mathcal{G}$ be a hereditary class of graphs. If every blowup of every prime graph in $\mathcal{G}$ is recolorable, then every graph in $\mathcal{G}$ is recolorable.
\end{theorem}

 We use the above result in Section \ref{p5free} along with the modular decomposition of several subclasses of $P_5$-free graphs to prove the following. (See Figure \ref{fig:u5} for graphs mentioned in Theorems \ref{p5diamond} and \ref{PHB}.)
 
\begin{theorem}\label{p5diamond}
    Every ($P_5$, diamond)-free graph is recolorable.
\end{theorem}
\begin{theorem}\label{PHB}
    Every ($P_5$, house, bull)-free graph is recolorable.
\end{theorem}

Note that there exist $P_5$-free graphs that are not recolorable \cite{2K2free}. A graph $G$ is said to be \emph{$P_4$-sparse} if every set of five vertices in $G$ induces at most one $P_4$. The class of $P_4$-sparse graphs generalizes the class of $P_4$-free graphs. Jamison and Olariu \cite{Jamison} gave a forbidden induced subgraph characterization for $P_4$-sparse graphs: A graph $G$ is $P_4$-sparse if and only if $G$ is ($P_5$, $C_5$, house, fork, co-fork, banner, co-banner)-free. In 1997, Fouquet and Giakoumakis \cite{Fouquet2} introduced the class of semi-$P_4$-sparse graphs which is a generalization of the class of $P_4$-sparse graphs. A graph $G$ is \emph{semi-$P_4$-sparse} if $G$ is ($P_5$, $C_5$, co-fork)-free. We prove the following.

\begin{theorem}\label{PHK}
    Every semi-$P_4$-sparse graph is recolorable.
\end{theorem}

Feghali and Fiala \cite{Feghali}, proved that every 3-colorable ($P_5$, house, $C_5$)-free graph is recolorable and asked if every ($P_5$, house, $C_5$)-free graph is recolorable. We prove the following.

\begin{theorem}\label{PHC}
    Every ($P_5$, house)-free graph is recolorable if and only if every ($P_5$, house, $C_5$)-free graph is recolorable.
\end{theorem}

Section \ref{conclusion} contains some open questions.

\section{Preliminaries}

 A \emph{path} in a graph is a sequence of distinct vertices $v_i$ and edges of the form $v_0$, $v_0v_1$, $v_1,\dots, v_{p-1}$, $v_{p-1}v_p$, $v_p$. The length of a path is equal to the number of edges in the path. A graph is said to be \emph{connected} if there exists a path between every pair of distinct vertices of the graph. A vertex $v\in V(G)$ is said to be \emph{universal} if $v$ is adjacent to every vertex in $V(G)\setminus \{v\}$. We say a coloring $\alpha$ of $G$ is \emph{reachable} from a coloring $\beta$ of $G$ in $R_{\ell}(G)$ if there is a path between $\alpha$ and $\beta$ in $R_{\ell}(G)$. Where $A\subseteq V(G)$ and  $B\subseteq V(G)$, we say $A$ is \emph{complete to} $B$ if every vertex in $A$ is adjacent to every vertex in $B$.  A set of mutually adjacent vertices in a graph is called a \emph{clique}. A set of mutually non-adjacent vertices in a graph is called an \emph{independent set}. Where $\alpha$ is a coloring of $G$ and $S\subseteq V(G)$, we use $\alpha(S)$ to denote the set of colors that appear on the vertices of $S$ under $\alpha$.
 
 Let $P_n$, $C_n$, and $K_n$ denote the path, cycle, and complete graph on $n$ vertices, respectively. Let $K_{p,q}$ denote the complete bipartite graph with $p$ vertices in one set of the bipartition and $q$ vertices in the other. A graph $H$ is said to be a \emph{subgraph} of $G$, if $V(H)\subseteq V(G)$ and $E(H)\subseteq E(G)$. A subgraph of $G$ \emph{induced} by a subset $S\subseteq V(G)$ is the subgraph of $G$ with vertex-set $S$ and edge-set all edges of $G$ which have both ends in $S$; it is denoted by $G[S]$. An induced subgraph $H$ of $G$ is said to be \emph{proper} if $V(H) \subset V(G)$. For a subset $S\subseteq V(G)$, we use $G$-$S$ to denote the subgraph of $G$ obtained by deleting the vertices of $S$ from $G$. For a vertex $v$, we use $G$-$v$ instead of $G$-$\{v\}$. For a subset $M\subseteq E(G)$, we use $G$-$M$ to denote the graph obtained by deleting $M$ from $G$, that is, the graph with vertex-set $V(G)$ and edge-set $E(G)\setminus M$.

  The \emph{complement of a graph} $G$, denoted \emph{co-$G$}, is the graph with the same vertices as $G$ such that two vertices are adjacent in co-$G$ if and only if they are non-adjacent in $G$. A \emph{component} of a graph $G$ is a maximal connected subgraph of $G$. For two vertex-disjoint graphs $G$ and $H$, the \emph{disjoint union} of $G$ and $H$, denoted by $G+H$, is the graph with vertex-set $V(G) \cup V(H)$ and edge-set $E(G) \cup E(H)$. For a positive integer $r$, we use $rG$ to denote the graph consisting of the disjoint union of $r$ copies of $G$. The \emph{join} of two graphs $G_1$ and $G_2$ is the graph obtained from the disjoint union of $G_1$ and $G_2$ by adding edges from each vertex of $G_1$ to every vertex of $G_2$. If $G$ is disconnected, then it is easy to see that $G$ is recolorable if and only if every component of $G$ is recolorable. So we may assume that $G$ is connected when appropriate. We also note that if $G$ is a join of two graphs $G_1$ and $G_2$, then $G$ is recolorable if both $G_1$ and $G_2$ are recolorable \cite{manuscript}.

\section{Complexity}\label{sec:comp}

% A non-empty set $S\subseteq V(G)$ is a \emph{module} if every vertex in $V(G)\setminus S$ is either adjacent to every vertex of $S$ or no vertex of $S$. A module $S\subset V(G)$ is said to be \emph{non-trivial} if $S$ contains at least two vertices. Note that for each $v\in V(G)$, $\{v\}$ is a trivial module in $G$. A graph $G$ is \emph{prime} if it does not contain any non-trivial module.

A module $S$ is said to be \emph{maximal} if $S\subset V(G)$ and it is not contained in a larger module $M\subset V(G)$. Two sets $X$ and $Y$ \emph{overlap} if $X$-$Y$, $Y$-$X$, and $X\cap Y$ are all non-empty. If two modules $S_1$ and $S_2$ of $G$ overlap, then it is easy to see that $S_1$-$S_2$, $S_2$-$S_1$, $S_1\cap S_2$, and $S_1\cup S_2$ are also modules of $G$. If $G$ is neither a join nor a disjoint union of two graphs, then any two maximal modules of $G$ are disjoint, and hence $V(G)$ can be partitioned into maximal modules $S_1$, $S_2$,\dots, $S_m$. Partitioning $V(G)$ into maximal modules can be done in linear-time; see \cite{Habib} for more information.

For a graph $G$, the sibling of $G$ is the graph obtained by adding, for each $x\in V(G)$, a vertex $x^{'}$ adjacent to only $x$.

\begin{lemma}\label{Hprime}
    For any connected graph $G$ the sibling of $G$ is prime.
\end{lemma}
\begin{proof}
    Let $H$ be the sibling of a connected graph $G$ obtained by adding, for each $x\in V(G)$, a vertex $x^{'}$ adjacent to only $x$. Assume $H$ is not prime. Then $H$ contains a non-trivial module $S^{'}$.
    
    Suppose $S^{'} \cap V(G)$ = $\emptyset$. Let $x^{'}$ and $y^{'}$ be in $S^{'}$. Since $x\in N_{H}(x^{'})\setminus N_{H}(y^{'})$, $x$ must be in $S^{'}$, a contradiction. Thus $S^{'} \cap V(G) \neq \emptyset$.\\
    \noindent
    \textit{Claim 1}: For any $x\in S^{'}\cap V(G)$, $x^{'}\in S^{'}$.\\
    Let $x\in S^{'}\cap V(G)$. Suppose $x^{'}\notin S^{'}$. Since $S^{'}$ is a non-trivial module, it must contain some vertex $y$ of $H$ different from $x$. Then $x^{'}$ is adjacent to $x$ but not to $y$, this contradicts that $S^{'}$ is a module. This proves Claim 1.\\
    \noindent
    \textit{Claim 2}: For any $x\in S^{'}\cap V(G)$, $N_{G}(x)\subseteq S^{'}$.\\
    Let $x\in S^{'}\cap V(G)$. By Claim 1, $x^{'}\in S^{'}\cap V(H)$. Since every vertex in $N_{G}(x)$ is adjacent to $x$ but not to $x^{'}$, we have $N_{G}(x)\subseteq S^{'}$. This proves Claim 2.

    Since $H$ and $G$ are connected and $V(H) = V(G) \cup \{x^{'} \mid x \in V(G)\}$, by Claims 1 and 2, we have $S^{'} = V(H)$. This contradicts that $S^{'}$ is a non-trivial module.    
\end{proof}

\begin{theorem}[\cite{Garey1974}]
    For all $k\ge 3$, deciding whether a connected graph is $k$-colorable is NP-complete.
\end{theorem}

We are now ready to prove Theorem \ref{thm:primecoloring}.
\begin{proof}[Proof of Theorem \ref{thm:primecoloring}]
    Clearly the problem is in NP. To prove that it is NP-complete we will give a reduction from a connected graph to a prime graph. Let $G$ be a connected graph and let $H$ be the sibling of $G$ obtained by adding, for each $x\in V(G)$, a vertex $x^{'}$ adjacent to only $x$. Note that, for all $k\ge 3$, $H$ is $k$-colorable if and only if $G$ is $k$-colorable. By Lemma \ref{Hprime}, $H$ is a prime graph. This proves the theorem.
\end{proof}

\begin{theorem}[\cite{Bonsma2009}]
    Given a connected graph $G$, for all $k\ge 4$, deciding whether there exists a path between two colorings in $R_{k}(G)$ is PSPACE-complete.
\end{theorem}

We will now prove Theorem \ref{thm:primerecoloring}.
\begin{proof}[Proof of Theorem \ref{thm:primerecoloring}]
    Clearly the problem is in PSPACE. To prove that it is PSPACE-complete we will give a reduction from a connected graph to a prime graph. Let $G$ be a connected graph and let $H$ be the sibling of $G$ obtained by adding, for each $x\in V(G)$, a vertex $x^{'}$ adjacent to only $x$. Then, by Lemma \ref{Hprime}, $H$ is a prime graph. Let $k\ge 4$. Let $\alpha$ and $\beta$ be any two $k$-colorings of $G$. Let $\alpha^{'}$ and $\beta^{'}$ be two $k$-colorings of $H$ obtained by extending $\alpha$ and $\beta$, respectively, by coloring each vertex $x^{'}\in V(H)\setminus V(G)$ a color different from the color of $x$.
    
    Since $G$ is an induced subgraph of $H$, any $k$-coloring of $H$ restricted to $V(G)$ is a $k$-coloring of $G$. Thus if there is a path between $\alpha^{'}$ and $\beta^{'}$ in $R_{k}(H)$, then there is a path between $\alpha$ and $\beta$ in $R_{k}(G)$.

    Assume that there is a path $P$ between $\alpha$ and $\beta$ in $R_{k}(G)$. Let $\gamma$ and $\psi$ be any two $k$-colorings of $G$ adjacent in the path $P$. Let $\gamma^{'}$ and $\psi^{'}$ be any two $k$-colorings of $H$ whose restriction to $V(G)$ is $\gamma$ and $\psi$, respectively. We will prove that there is a path between $\gamma^{'}$ and $\psi{'}$ in $R_{k}(H)$. Since $\gamma$ and $\psi$ are adjacent in the path $P$, there exists a unique vertex $y\in V(G)$ such that $\gamma(y) \neq \psi(y)$. Let $y^{'}\in V(H)\setminus V(G)$ be adjacent to only $y$. Starting from $\gamma^{'}$, recolor $y^{'}$ with a color not in $\{\gamma(y), \psi(y)\}$. Recolor the vertex $y$ with color $\psi(y)$. Recolor each vertex $x^{'}$ in $V(G)\setminus V(H)$, one at a time, with color $\psi^{'}(x^{'})$ to obtain the coloring $\psi^{'}$.

    Hence there is a path between $\alpha$ and $\beta$ in $R_{k}(G)$ if and only if there is a path between $\alpha^{'}$ and $\beta^{'}$ in $R_{k}(H)$. This completes the proof.
\end{proof}

\section{Recoloring and Modular Decomposition}\label{modulardecomposition}

Let $G$ be a graph which is neither a join nor a disjoint union of two graphs; the \emph{skeleton} of $G$, denoted $G^{*}$, is the graph obtained by contracting each maximal module to a single vertex (and removing any loops and multiple edges created). It can also be obtained from $G$ by recursively substituting a vertex for each maximal module in $G$. Note that $G^{*}$ is a prime subgraph of $G$, induced by a set containing exactly one vertex from each maximal module in $G$.

For the remainder of this section, we use the following decomposition. Let $G$ be a graph which is neither a join nor a disjoint union of two graphs. Let $S_1$,\dots, $S_m$ be a partition of $V(G)$ into maximal modules. Note that $S_p \cap S_q$ = $\emptyset$ for all $p\neq q$, $p$ and $q \in [m]$. Let $\chi(G) = k$ and let $\chi(G[S_p]) = k_p$, for all $p\in [m]$. For a graph $G$, the \emph{clique skeleton} of $G$ is the graph obtained from $G$ by recursively substituting a clique $Q_p$ of size $k_p$ for each $S_p$ in $G$, where $p \in [m]$. We use $G(Q_1,\dots,Q_m)$ to denote the clique skeleton of $G$. Similar operations have been defined in the literature, for example see \cite{chinh}. In Figure \ref{fig:moddecomp}, for a graph $G$, we have illustrated its skeleton and its clique skeleton.

Let $H$ = $G(Q_1,\dots,Q_m)$ be the clique skeleton of $G$. Then for any $k$-coloring $\alpha$ of $G$, there exists a $k$-coloring $\beta$ of $H$ such that $\beta(Q_p)\subseteq \alpha(S_p)$, for all $p\in [m]$. Note that such a coloring $\beta$ of $H$ can be obtained by coloring each $Q_p$ with colors from the set $\alpha(S_p)$, for all $p\in [m]$. Similarly, given any $k$-coloring $\beta$ of $H$, there exists a $k$-coloring $\alpha$ of $G$ such that for all $p\in [m]$, $\alpha(S_p)$ = $\beta(Q_p)$. Thus $\chi(G)$ = $\chi(H)$. Note that the clique skeleton of a graph $G$ is a blowup of the skeleton of $G$. 

\begin{figure}[ht]
    \centering
    \definecolor{c1}{rgb}{1,0.4,0.4}
    \definecolor{c2}{rgb}{0.8,0.7,0.6}
    \definecolor{c3}{rgb}{0.5,0.9,0.6}
    \definecolor{c4}{rgb}{0.7,0.8,0.4}
    \definecolor{c5}{rgb}{0.6,0.7,0.9}
    \subfloat[A graph $G$]{
    \begin{tikzpicture}
    \node[draw,circle,fill=c1](1) at (0,1.25) {$a_1$};
    \node[draw,circle,fill=c1](2) at (0,-0.25) {$a_2$};
    \node[draw,circle,fill=c2](3) at (2,1.25) {$b_1$};
    \node[draw,circle,fill=c2](4) at (3,-0.25) {$b_2$};
    \node[draw,circle,fill=c3](5) at (3,3) {$c_1$};
    \node[draw,circle,fill=c4](6) at (5,0.5) {$d_1$};
    \node[draw,circle,fill=c5](7) at (7.5,3) {$e_1$};
    \node[draw,circle,fill=c5](8) at (7.5,1.75) {$e_2$};
    \node[draw,circle,fill=c5](9) at (7.5,0.5) {$e_3$};
    \node[draw,circle,fill=c5](10) at (7.5,-0.75) {$e_4$};
    \node[draw,circle,fill=c5](11) at (7.5,-2) {$e_5$};
    \draw (1)--(2);\draw (1)--(3);\draw (1)--(4);\draw (2)--(3);\draw (2)--(4);\draw (5)--(3);\draw (4)--(5);\draw (6)--(3);\draw (6)--(4);\draw (5)--(6);\draw (7)--(6);\draw (8)--(6);\draw (9)--(6);\draw (10)--(6);\draw (11)--(6);\draw (11)--(10);\draw (10)--(9);\draw (9)--(8);\draw (8)--(7);\draw (11) .. controls +(right:1.5cm) and +(right:1.5cm) .. (7);
    \end{tikzpicture}}\\ \vspace{0.5cm}
    \subfloat[The skeleton of $G$]{
    \begin{tikzpicture}[scale=0.8]
    \node[draw,circle,fill=c1](1) at (1,0.5) {$A$};
    \node[draw,circle,fill=c2](4) at (3,0.5) {$B$};
    \node[draw,circle,fill=c3](5) at (4,2) {$C$};
    \node[draw,circle,fill=c4](6) at (5,0.5) {$D$};
    \node[draw,circle,fill=c5](9) at (7,0.5) {$E$};
    \draw (1)--(4);\draw (4)--(5);\draw (6)--(4);\draw (5)--(6);\draw (9)--(6);
    \end{tikzpicture}}\hspace{8mm}
    \subfloat[The clique skeleton of $G$]{
    \begin{tikzpicture}[scale=0.8]
    \node[draw,circle,fill=c1](1) at (0.5,1.25) {$A_1$};
    \node[draw,circle,fill=c1](2) at (0.5,-0.25) {$A_2$};
    \node[draw,circle,fill=c2](4) at (3,0.5) {$B$};
    \node[draw,circle,fill=c3](5) at (4,2) {$C$};
    \node[draw,circle,fill=c4](6) at (5,0.5) {$D$};
    \node[draw,circle,fill=c5](8) at (7,2) {$E_1$};
    \node[draw,circle,fill=c5](9) at (7,0.5) {$E_2$};
    \node[draw,circle,fill=c5](10) at (7,-1) {$E_3$};
    \draw (1)--(2);\draw (1)--(4);\draw (2)--(4);\draw (4)--(5);\draw (6)--(4);\draw (5)--(6);\draw (8)--(6);\draw (9)--(6);\draw (10)--(6);\draw (10)--(9);\draw (9)--(8);\draw (8) .. controls +(right:1cm) and +(right:1cm) .. (10);
    \end{tikzpicture}}    
    \caption{A graph $G$, its skeleton, and its clique skeleton.}
    \label{fig:moddecomp}
\end{figure}

\begin{lemma}\label{homo1}
    Let $G$ be a graph such that $V(G)$ can be partitioned into maximal modules $S_1$,\dots, $S_m$. Assume every proper induced subgraph of $G$ is recolorable. Let $\alpha$ and $\beta$ be two $\chi$-colorings of $G$ such that for each $S_p$ of $G$, $\alpha(S_p)\subseteq \beta(S_p)$. Then there is a path between $\alpha$ and $\beta$ in $R_{\ell}(G)$, for all $\ell\geq \chi(G)$+1.
\end{lemma}
\begin{proof}
    Let $G$ be a graph such that $V(G)$ can be partitioned into maximal modules $S_1$,\dots, $S_m$. Let $\gamma$ be a coloring of $G$ and let $p\in [m]$. For any vertex $v\in S_p$, since its neighbors not in $S_p$ are complete to $S_p$, they will not be colored with a color from the set $\gamma(S_p)$ under $\gamma$. Let $\ell\ge \chi(G)$+1 and let $\{1,\dots,\ell\}$ be the set of available colors. Let $\alpha$ and $\beta$ be two $\chi$-colorings of $G$ such that $\alpha(S_p)\subseteq \beta(S_p)$, for all $p\in [m]$. Then there exists a color, say $c$, that does not appear on $V(G)$ under either $\alpha$ or $\beta$. Since $G[S_p]$ is recolorable for all $p\in [m]$, we can use the color $c$ to recolor the vertices in $S_p$ from $\alpha$ to $\beta$, for each $p\in [m]$ one at a time, without changing the colors on vertices in $S_j$ where $j\neq p$.
\end{proof}

\begin{lemma}\label{homo1A}
    Let $G$ be a graph such that $V(G)$ can be partitioned into maximal modules $S_1$,\dots, $S_m$. Let $H$ = $G(Q_1,\dots,Q_m)$ be the clique skeleton of $G$. Let $\alpha$ be any $\chi$-coloring of $G$. Let $\beta_1$ and $\beta_2$ be any two $\chi$-colorings of $H$ such that for each $S_p$ of $G$, $\beta_1(Q_p)\subseteq \alpha(S_p)$ and $\beta_2(Q_p)\subseteq \alpha(S_p)$. Then there is a path between $\beta_1$ and $\beta_2$ in $R_{\ell}(H)$, for all $\ell\geq \chi(H)$+1.
\end{lemma}
\begin{proof}
    Assume the hypotheses. Let $\ell \geq \chi(H)$+1. Let $\{1,\dots,\ell\}$ be the set of available colors. Since $\alpha$ is $\chi$-coloring of $G$, there exists a color, say c, that does not appear on $V(G)$ under $\alpha$. Since $\beta_1(Q_p)\subseteq \alpha(S_p)$ and $\beta_2(Q_p)\subseteq \alpha(S_p)$, for all $p\in [m]$, the color $c$ also does not appear on $V(H)$ under either $\beta_1$ or $\beta_2$. Since each $Q_p$ is a clique, it is recolorable. Recolor the vertices in $Q_p$ from $\beta_1$ to $\beta_2$ using the extra color $c$, for each $p\in [m]$ one at a time, without changing the colors on vertices in $Q_j$ where $j\neq p$.
\end{proof}

\begin{lemma}\label{homo2}
    Let $G$ be a graph such that $V(G)$ can be partitioned into maximal modules $S_1$,\dots, $S_m$. Let $H$ = $G(Q_1,\dots,Q_m)$ be the clique skeleton of $G$. Let $\ell\geq \chi(G)$+1. Let $\alpha$ be an $\ell$-coloring of $G$ and let $\beta$ be a $\chi$-coloring of $G$. Let $\alpha^{'}$ be an $\ell$-coloring of $H$ such that for all $p\in [m]$, $\alpha^{'}(Q_{p}) \subseteq \alpha(S_{p})$. Let $\beta^{'}$ be a $\chi$-coloring of $H$ such that for all $p\in [m]$, $\beta^{'}(Q_{p})\subseteq \beta(S_{p})$. If there is a path between $\alpha$ and $\beta$ in $R_{\ell}(G)$, then there is a path between $\alpha^{'}$ and $\beta^{'}$ in $R_{\ell}(H)$.
\end{lemma}
\begin{proof}
    Assume the hypotheses. Let $\mathbb{P}$ be a path between $\alpha$ and $\beta$ in $R_{\ell}(G)$. Let $\alpha_i$ and $\alpha_{i+1}$ be two $\ell$-colorings of $G$, which are adjacent in $\mathbb{P}$. Let $\alpha_{i}^{'}$ be any $\ell$-coloring of $H$ such that for all $p \in [m]$, $\alpha_{i}^{'}(Q_p) \subseteq \alpha_{i}(S_p)$. We prove that there is a path from $\alpha_{i}^{'}$ to some $\ell$-coloring of $H$, say $\alpha_{i+1}^{'}$, in $R_{\ell}(H)$ such that for all $p \in [m]$, $\alpha_{i+1}^{'}(Q_p) \subseteq \alpha_{i+1}(S_p)$. 
    
    Since $\alpha_i$ and $\alpha_{i+1}$ are adjacent in $R_{\ell}(G)$, they differ on a single vertex, say $v$. Without loss of generality, let $v\in S_j$ for some $j\in [m]$. Note that there are no vertices in the closed neighborhood of $v$ colored $\alpha_{i+1}(v)$ under the coloring $\alpha_i$ of $G$. Thus, there are no vertices in the closed neighborhood of any vertex in $Q_j$ colored $\alpha_{i+1}(v)$ under the coloring $\alpha_i^{'}$ of $H$. Let $\alpha_{i+1}^{'}$ be the coloring obtained from $\alpha_{i}^{'}$ by recoloring the vertex in $Q_j$ colored $\alpha_i(v)$ with the color $\alpha_{i+1}(v)$. It follows that there is a path between $\alpha_{i}^{'}$ and $\alpha_{i+1}^{'}$ in $R_{\ell}(H)$ such that for all $p \in [m]$, $\alpha_{i+1}^{'}(Q_p) \subseteq \alpha_{i+1}(S_p)$. Thus, if there exists a path between $\alpha$ and $\beta$ in $R_{\ell}(G)$, then there exists a path from $\alpha^{'}$ to some coloring $\beta^{*}$ in $R_{\ell}(H)$ such that for all $p\in [m]$, $\beta^{*}(Q_p)\subseteq \beta(S_p)$. Since $\beta^{*}$ and $\beta^{'}$ are two $\chi$-colorings of $H$ such that for all $p\in [m]$, $\beta^{*}(Q_p)\subseteq \beta(S_p)$ and $\beta^{'}(Q_p)\subseteq \beta(S_p)$, by Lemma \ref{homo1A} there is a path between $\beta^{*}$ and $\beta^{'}$ in $R_{\ell}(H)$. Therefore, there is a path between $\alpha^{'}$ and $\beta^{'}$ in $R_{\ell}(H)$.
\end{proof}

% \begin{lemma}
%     Let $G$ be a graph such that every proper induced subgraph is recolorable. Let $H$ be the clique skeleton of $G$. Let $\alpha$ be a $\chi$-coloring of $H$. Let $\beta_1$ and $\beta_2$ be two $\chi$-colorings of $G$ such that $\beta_1(S_p)\supseteq \alpha(Q_p)$ and $\beta_2(S_p)$ = $\alpha(Q_p)$, for all $p\in [m]$. Then there is a path between $\beta_1$ and $\beta_2$ in $R_{\ell}(G)$, for all $\ell\geq \chi(G)$+1.
% \end{lemma}

\begin{lemma}\label{homo3}
    Let $G$ be a graph such that $V(G)$ can be partitioned into maximal modules $S_1$,\dots, $S_m$. Assume every proper induced subgraph of $G$ is recolorable. Let $H$ = $G(Q_1,\dots,Q_m)$ be the clique skeleton of $G$. Let $\ell\geq \chi(G)$+1. Let $\alpha$ and $\beta$ be two $\chi$-colorings of $G$. Let $\alpha^{'}$ and $\beta^{'}$ be two $\chi$-colorings of $H$ such that for all $p\in [m]$, $\alpha^{'}(Q_p)\subseteq \alpha(S_p)$ and $\beta^{'}(Q_p)\subseteq \beta(S_p)$. Then there exists a path between $\alpha$ and $\beta$ in $R_{\ell}(G)$ if and only if there exists a path between $\alpha^{'}$ and $\beta^{'}$ in $R_{\ell}(H)$.
\end{lemma}
\begin{proof}
    Assume the hypotheses. 
    
    First assume that there is a path between $\alpha$ and $\beta$ in $R_{\ell}(G)$. Then by Lemma \ref{homo2}, there is a path between $\alpha^{'}$ and $\beta^{'}$ in $R_{\ell}(H)$. 
    
    Now assume that there is a path between $\alpha^{'}$ and $\beta^{'}$ in $R_{\ell}(H)$. Let $\alpha^{*}$ and $\beta^{*}$ be two $\chi$-colorings of $G$ such that for all $p\in [m]$, $\alpha^{*}(S_p)$ = $\alpha^{'}(Q_p)$ and $\beta^{*}(S_p)$ = $\beta^{'}(Q_p)$. Then since for all $p\in [m]$, $\alpha^{'}(Q_p)\subseteq \alpha(S_p)$ and $\beta^{'}(Q_p)\subseteq \beta(S_p)$, clearly for all $p\in [m]$, $\alpha^{*}(S_p)\subseteq \alpha(S_p)$ and $\beta^{*}(S_p)\subseteq \beta(S_p)$. Thus by Lemma \ref{homo1}, there is a path from $\alpha$ to $\alpha^{*}$ and there is a path from $\beta$ to $\beta^{*}$ in $R_{\ell}(G)$, for any $\ell\geq \chi(G)$+1. We will now prove that there is a path from $\alpha^{*}$ to $\beta^{*}$ in $R_{\ell}(G)$.

    Let $\mathbb{P}$ be a path between $\alpha^{'}$ and $\beta^{'}$ in $R_{\ell}(H)$. Let $\alpha_{i}^{'}$ and $\alpha_{i+1}^{'}$ be two $\ell$-colorings of $H$ which are adjacent in $\mathbb{P}$. Let $\alpha_{i}$ be an $\ell$-coloring of $G$ such that for all $p\in [m]$, $\alpha_{i}(S_p) = \alpha_{i}^{'}(Q_p)$. We prove that there is a path from $\alpha_{i}$ to some $\ell$-coloring of $G$, say $\alpha_{i+1}$, in $R_{\ell}(G)$, such that $\alpha_{i+1}(S_p) $ = $ \alpha_{i+1}^{'}(Q_p)$, for all $p \in [m]$. Since $\alpha_{i}^{'}$ and $\alpha_{i+1}^{'}$ are adjacent in $R_{\ell}(H)$, they differ on a single vertex. Let $v \in Q_j \subseteq V(H)$, where $\alpha_{i}^{'}(v) \neq \alpha_{i+1}^{'}(v)$. Note that there are no vertices in the closed neighborhood of $v$ colored $\alpha_{i+1}^{'}(v)$ under the coloring $\alpha_i^{'}$ of $H$. Thus there are no vertices in the closed neighborhood of any vertex in $S_j$ colored $\alpha_{i+1}^{'}(v)$ under the coloring $\alpha_i$ of $G$. Let $\alpha_{i+1}$ be an $\ell$-coloring of $G$ obtained from $\alpha_i$ by recoloring every vertex in $S_j$ colored $\alpha_{i}^{'}(v)$ with the color $\alpha_{i+1}^{'}(v)$, one at a time. This corresponds to a path between $\alpha_{i}$ and $\alpha_{i+1}$ in $R_{\ell}(G)$, where $\alpha_{i+1}(S_p) $ = $ \alpha_{i+1}^{'}(Q_p)$ for all $p \in [m]$. Hence, if there is a path between $\alpha^{'}$ and $\beta^{'}$ in $R_{\ell}(H)$, then there is a path from $\alpha^{*}$ to some coloring $\gamma$ in $R_{\ell}(G)$, such that for all $p\in [m]$, $\gamma(S_p)$ = $\beta^{'}(Q_p)$. Since $\beta^{*}$ and $\gamma$ are two $\chi$-colorings of $G$ such that for all $p\in [m]$, $\beta^{*}(S_p)$ = $\gamma(S_p)$, by Lemma \ref{homo1}, there is a path between $\beta^{*}$ and $\gamma$ in $R_{\ell}(G)$. Therefore, there is a path between $\alpha^{*}$ and $\beta^{*}$ in $R_{\ell}(G)$.
\end{proof}

\begin{theorem}\label{moddecomposition}
    Let $G$ be a graph such that $V(G)$ can be partitioned into maximal modules $S_1$,\dots, $S_m$. Let $H$ = $G(Q_1,\dots,Q_m)$ be the clique skeleton of $G$. Let $\ell\geq \chi(G)$+1. If every proper induced subgraph of $G$ is recolorable, then $R_{\ell}(G)$ is connected if and only if $R_{\ell}(H)$ is connected.
\end{theorem}
\begin{proof}
    Assume the hypotheses. 
    
    We first prove that, for any $\ell\geq \chi(G)$+1, if $R_{\ell}(G)$ is connected, then $R_{\ell}(H)$ is connected. Let $\alpha^{'}$ be any $\ell$-coloring of $H$ and let $\beta^{'}$ be a $\chi$-coloring of $H$. Let $\alpha$ be an $\ell$-coloring of $G$ and let $\beta$ be a $\chi$-coloring of $G$, such that for all $p\in [m]$, $\alpha(S_p)$ = $\alpha^{'}(Q_p)$ and $\beta(S_p)$ = $\beta^{'}(Q_p)$. Since $R_{\ell}(G)$ is connected, there is a path between $\alpha$ and $\beta$ in $R_{\ell}(G)$. Then by Lemma \ref{homo2}, there is a path between $\alpha^{'}$ and $\beta^{'}$ in $R_{\ell}(H)$. Therefore, $R_{\ell}(H)$ is connected.

    We now prove that, for any $\ell\geq \chi(G)$+1, if $R_{\ell}(H)$ is connected, then $R_{\ell}(G)$ is connected. Note that for each $\chi$-coloring $\mu$ of $G$ there is a $\chi$-coloring $\mu^{'}$ of $H$ such that for all $p\in [m]$, $\mu^{'}(Q_p)\subseteq \mu(S_p)$. Similarly, for each $\chi$-coloring $\mu^{'}$ of $H$ there is a $\chi$-coloring $\mu$ of $G$ such that for all $p\in [m]$, $\mu(S_p)$ = $\mu^{'}(Q_p)$. Hence, if $R_{\ell}(H)$ is connected, then by Lemma \ref{homo3}, there is a path between any two $\chi$-colorings of $G$ in $R_{\ell}(G)$. Thus, to prove that $R_{\ell}(G)$ is connected, it is sufficient to prove that given any $\ell$-coloring of $G$, we can reach a $\chi$-coloring in $R_{\ell}(G)$.

     Let $\alpha$ be an $\ell$-coloring of $G$. Since for all $p \in [m]$, $G[S_p]$ is recolorable, recolor the vertices in each $S_p$, one at a time, to use at most $\chi(G[S_p])$ colors without changing the colors on vertices in $S_j$, $j\neq p$, to obtain a coloring $\beta$ of $G$. Let $\beta^{'}$ be an $\ell$-coloring of $H$ such that for all $p \in [m]$, $\beta^{'}(Q_p) = \beta(S_p)$. Let $\gamma^{'}$ be a $\chi$-coloring of $H$. Since $R_{\ell}(H)$ is connected, there is a path $\mathbb{P}$ between $\beta^{'}$ and $\gamma^{'}$ in $R_{\ell}(H)$. Let $\alpha_{i}^{'}$ and $\alpha_{i+1}^{'}$ be two $\ell$-colorings of $H$ which are adjacent in the path $\mathbb{P}$. Let $\alpha_{i}$ be an $\ell$-coloring of $G$ such that for all $p \in [m]$, $\alpha_{i}(S_p) = \alpha_{i}^{'}(Q_p)$. We prove that there is a path from $\alpha_{i}$ to some $\ell$-coloring of $G$, say $\alpha_{i+1}$, in $R_{\ell}(G)$, such that for all $p \in [m]$, $\alpha_{i+1}(S_p) $ = $ \alpha_{i+1}^{'}(Q_p)$. Since $\alpha_{i}^{'}$ and $\alpha_{i+1}^{'}$ are adjacent in $R_{\ell}(H)$, they differ on a single vertex, say $v$. Let $v \in Q_j$, for some $j\in [m]$. Note that there are no vertices in the closed neighborhood of $v$ colored $\alpha_{i+1}^{'}(v)$ under the coloring $\alpha_{i}^{'}$ of $H$. Thus there are no vertices in the closed neighborhood of any vertex in $S_j$ colored $\alpha_{i+1}^{'}(v)$ under the coloring $\alpha_i$ of $G$. Let $\alpha_{i+1}$ be an $\ell$-coloring of $G$ obtained from $\alpha_i$ by recoloring every vertex in $S_j$ colored $\alpha_{i}^{'}(v)$, one at a time, with the color $\alpha_{i+1}^{'}(v)$. This corresponds to a path between $\alpha_{i}$ and $\alpha_{i+1}$ in $R_{\ell}(G)$, where $\alpha_{i+1}(S_p) = \alpha_{i+1}^{'}(Q_p)$ for all $p \in [m]$. Hence, if there is a path from $\beta^{'}$ to $\gamma^{'}$ in $R_{\ell}(H)$, then we can recolor the vertices of $G$ from any $\ell$-coloring $\alpha$ of $G$ to some $\chi$-coloring of $G$.
\end{proof}

The results below are well-known, for example see \cite{manuscript, biedl2021}.

\begin{lemma}\label{join}
    If $G$ is a join of two recolorable graphs $G_1$ and $G_2$, then $G$ is recolorable.
\end{lemma}

\begin{lemma}\label{independent}
Let $G$ be a graph with non-adjacent vertices $u$ and $v$, such that $N(u)\subseteq N(v)$. Then either\\
(i) ~$G$ is recolorable or\\
(ii) $G$-$u$ is not recolorable.
\end{lemma}

We now prove Theorem \ref{hereditary}.

\begin{proof}[Proof of Theorem \ref{hereditary}]
    Let $\mathcal{G}$ be a hereditary class of graphs and let $\mathcal{H}$ be the set of prime graphs in $\mathcal{G}$. Assume that for all $H\in \mathcal{H}$, every blowup of $H$ is recolorable. If possible, let $G\in \mathcal{G}$ be a vertex-minimal graph that is not recolorable. Then, by Lemma \ref{join}, $G$ is neither a join nor a disjoint union of two graphs. Let $S_1$,\dots, $S_m$ be a partition of $V(G)$ into maximal modules. Let $G^{*}$ be the skeleton of $G$. Then $G^{*}$ is a prime induced subgraph of $G$ and hence is in $\mathcal{H}$. Since $G^{*} \in \mathcal{H}$, from the hypothesis, every blowup of $G^{*}$ is recolorable. Let $H^{*}$ = $G(Q_1,\dots,Q_m)$ be the clique skeleton of $G$. Note that $H^{*}$ is a blowup of $G^{*}$ and hence is recolorable. By the vertex minimality of $G$, every proper induced subgraph of $G$ is recolorable. Therefore, since $H^{*}$ is recolorable, by Theorem \ref{moddecomposition}, $G$ is recolorable, a contradiction. 
 \end{proof}

 Let $\mathcal{G}$ be a hereditary class of graphs. Then Theorem \ref{hereditary} states that to determine if every graph in $\mathcal{G}$ is recolorable, it is sufficient to determine if every blowup of every prime graph in $\mathcal{G}$ is recolorable. We can further ask: To determine if every graph in $\mathcal{G}$ is recolorable, is it sufficient to determine if every prime graph in $\mathcal{G}$ is recolorable? We answer this in the positive for $2K_2$-free graphs and diamond-free graphs.

 \begin{theorem}
     Let $\mathcal{G}$ be a hereditary class of graphs such that every $G\in \mathcal{G}$ is 2$K_2$-free. If every prime graph in $\mathcal{G}$ is recolorable, then every graph in $\mathcal{G}$ is recolorable.
 \end{theorem}
 \begin{proof}
     Assume the hypothesis. If possible, let $G\in \mathcal{G}$ be a vertex-minimal graph that is not recolorable. Then $G$ is not prime. Since $G$ is not prime, it contains a maximal module with at least two vertices, say $S$. Let $u,v \in S$. If $S$ is an independent set, then $N(u)$ = $N(v)$. Then by Lemma \ref{independent}, $G$ is not vertex-minimal, a contradiction. So $S$ must contain an edge. Without loss of generality, let $uv \in E(G)$. If $S$ is complete to $V(G)\setminus S$, then $G$ is a join of $G[S]$ and $G$-$S$. Then by Lemma \ref{join}, $G$ is not vertex-minimal, a contradiction. Hence there is a vertex $x\in V(G)\setminus S$ that is non-adjacent to some vertex in $S$. Since $S$ is a module, $x$ is non-adjacent to every vertex in $S$. If $N(x)\subseteq N(v)$, then by Lemma \ref{independent}, $G$ is not vertex-minimal, a contradiction. Thus there is a vertex $y\in N(x)\setminus N(v)$. Then $y\in V(G)\setminus S$ since $x$ is non-adjacent to every vertex in $S$. Since $S$ is a module and $y$ is non-adjacent to $v\in S$, $y$ is non-adjacent to every vertex in $S$. Then $\{u,v,x,y\}$ induces a 2$K_2$, a contradiction. 
 \end{proof}

\begin{lemma}[\cite{manuscript}]\label{degree}
Let $G$ contain a vertex $v$ of degree at most $\chi(G)$-1. Then, either\\
(i) ~$G$ is recolorable or\\
(ii) $G$-$v$ is not recolorable.
\end{lemma}

\begin{theorem}\label{b2k2}
    Let $\mathcal{G}$ be a hereditary class of graphs such that every $G\in \mathcal{G}$ is diamond-free. If every prime graph in $\mathcal{G}$ is recolorable, then every graph in $\mathcal{G}$ is recolorable.
\end{theorem}
\begin{proof}
    Assume the hypothesis. If possible, let $G\in \mathcal{G}$ be a vertex-minimal graph that is not recolorable. Then $G$ is connected and not prime. Since $G$ is not prime, it contains a non-trivial module, say $S$, with at least two vertices.
    
    Since $G$ is connected, there exists a vertex $z\in V(G)\setminus S$ adjacent to every vertex in $S$. If $G[S]$ contains an induced $P_3$, induced by, say, the set $P\subseteq S$, then $P\cup \{z\}$ induces a diamond in $G$, a contradiction. So $G[S]$ is $P_3$-free and hence $G[S]$ is a disjoint union of cliques.
    
    If there are adjacent vertices $u$ and $v$ in $S$, then we claim that $N[u] = N[v]$ is a clique. Since $S$ is a module and a disjoint union of cliques, $N[u] = N[v]$. If $N[u]$ is not a clique, then there exist two non-adjacent vertices $x$ and $y$ in $V(G)\setminus S$ adjacent to both $u$ and $v$, a contradiction to the fact that $G$ is diamond-free. Thus $N[u] = N[v]$ is a clique. Therefore, where $d(u)$ denotes the degree of the vertex $u$ in $G$, $d(u)\le \omega(G)$-1 $\le \chi(G)$-1. By vertex minimality of $G$, the graph $G$-$u$ must be recolorable. Then by Lemma \ref{degree}, $G$ is recolorable, a contradiction.

    If there are no adjacent vertices in $S$, then $S$ is an independent set in $G$. Let $\{u, v\}\subseteq S$. Since $S$ is a module, $N(u)$ = $N(v)$. By vertex minimality of $G$, the graph $G$-$u$ must be recolorable. Then by Lemma \ref{independent}, $G$ is recolorable, a contradiction.
\end{proof}

\section{\texorpdfstring{$P_5$}\ -free graphs}\label{p5free}

In this section, we use the results from previous section to prove recolorability results for several subclasses of $P_5$-free graphs. A \textit{clique cutset} $Q$ of a graph $G$ is a clique in $G$ such that $G$-$Q$ has more components than $G$. A clique cutset $Q$ of a graph $G$ is called a \textit{tight clique cutset} if there exists a component $H$ of $G$-$Q$ which is complete to $Q$.

We use the following results in our proof.
\begin{lemma}[\cite{Belavadi}]\label{tight}
    Suppose $G$ has a tight clique cutset. If every proper induced subgraph of $G$ is recolorable, then $G$ is recolorable.
\end{lemma}

\begin{lemma}\label{blow3K1-free}
    Every blowup of a 3$K_1$-free graph is recolorable.
\end{lemma}
\begin{proof}
    Clearly, every blowup of a 3$K_1$-free graph is 3$K_1$-free. Thus the lemma follows from the result proved in \cite{3k1free} that every 3$K_1$-free graph is recolorable.
\end{proof}

A graph is said to be \emph{chordal} if it does not contain an induced cycle of length more than three.

\begin{lemma}\label{bchordal}
    Every blowup of a chordal graph is recolorable.
\end{lemma}
\begin{proof}
    Clearly, every blowup of a chordal graph is chordal. Thus the lemma follows from the result proved in \cite{bonamy2014} that every chordal graph is recolorable.
\end{proof}

\begin{lemma}\label{5vertex}
    Every blowup of a prime graph with at most 5 vertices is recolorable.
\end{lemma}
\begin{proof}
    Let $G$ be any prime graph with at most 5 vertices. If $G$ is ($P_5$, $C_4$)-free, then every blowup of $G$ is ($P_5$, $C_4$)-free. Thus the lemma follows from the result proved in \cite{Belavadi} that every ($P_5$, $C_4$)-free graph is recolorable. Note that there are no prime graphs with at most 4 vertices containing an induced $P_5$ or an induced $C_4$. 
    
    The only 5-vertex prime graphs containing an induced $P_5$ or an induced $C_4$ are $P_5$ and the house. If $G$ is the house, then since the house is 3$K_1$-free, the result follows from Lemma \ref{blow3K1-free}. If $G$ is $P_5$, then since $P_5$ is chordal, the result follows from Lemma \ref{bchordal}.
\end{proof}

Two edges in a graph are said to be \emph{independent} if they do not share an end vertex. A \emph{matching} is a set of mutually independent edges. A \emph{matched co-bipartite graph} is the complement of $K_{p,q}$-$M$, where $M$ is a maximum matching, for some $p,q \ge 1$. Note that every matched co-bipartite graph is a co-bipartite graph. We use the modular decomposition of ($P_5$, diamond)-free graphs given by Brandst\"{a}dt \cite{Andreas} to prove that every ($P_5$, diamond)-free graph is recolorable.

\begin{theorem}[\cite{Andreas}]\label{structure}
    If $G$ is a prime ($P_5$, diamond)-free graph containing a $2K_2$, then $G$ is a matched co-bipartite graph.
\end{theorem}

\begin{theorem}[\cite{manuscript}]\label{2K2dia-free}
    Every (2$K_2$, diamond)-free graph is recolorable.
\end{theorem}

We now have all the ingredients to prove Theorem \ref{p5diamond}.

\begin{proof}[Proof of Theorem \ref{p5diamond}]
    By Theorem \ref{b2k2}, it is sufficient to prove that every prime ($P_5$, diamond)-free graph is recolorable. Let $G$ be a prime ($P_5$, diamond)-free graph. If $G$ is 2$K_2$-free, then the result follows from Theorem \ref{2K2dia-free}. If $G$ contains a $2K_2$, then by Theorem \ref{structure}, $G$ is a matched co-bipartite graph. Since matched co-bipartite graphs are 3$K_1$-free, by Lemma \ref{blow3K1-free}, $G$ is recolorable.
\end{proof}

\begin{lemma}[\cite{Fouquet1}]\label{p5bip}
    Let $G$ = ($V$, $E$) be a prime graph on $n$ vertices. Then $G$ is bipartite and $P_5$-free if and only if the following conditions hold:
    \begin{itemize}
        \item There exists a partition of $V(G)$ into two independent sets $B$ = $\{ b_1,b_2,\dots,b_{\frac{n}{2}} \}$ and $W$ = $\{w_1,w_2,\dots,w_{\frac{n}{2}} \}$.
        \item The neighbors of $b_i$ ($i$ = 1,$\dots ,\frac{n}{2}$) are precisely $w_1,\dots,w_{\frac{n}{2}-i+1}$.
    \end{itemize}
\end{lemma}

\begin{lemma}\label{blowupbip}
    Every blowup of a prime $P_5$-free bipartite graph is recolorable.
\end{lemma}
\begin{proof}
    If possible, let $H$ be a vertex-minimal minimal counterexample, say $H$ is a blowup of some prime $P_5$-free bipartite graph $G$. By Lemma \ref{p5bip}, $G$ has a vertex of degree 1. The neighbor of this degree 1 vertex is a tight clique cutset of $G$. Thus $H$ has a tight clique cutset. Every proper induced subgraph of $H$ is also a blowup of some prime $P_5$-free bipartite graph and hence recolorable. Then by Lemma \ref{tight}, $H$ is recolorable, a contradiction.
\end{proof}

\begin{lemma}\label{blowupco-bip}
    Every blowup of a co-bipartite graph is recolorable.
\end{lemma}
\begin{proof}
    Clearly, every co-bipartite graph is 3$K_1$-free. Hence the result follows from Lemma \ref{blow3K1-free}.
\end{proof}

\begin{theorem}[\cite{Fouquet}]\label{strPHB}
    If $G$ is a prime ($P_5$, house, bull)-free graph on 6 or more vertices without a universal vertex, then $G$ is either a bipartite graph or a co-bipartite graph.
\end{theorem}

We now prove Theorem \ref{PHB}.

\begin{proof}[Proof of Theorem \ref{PHB}]
    By Theorem \ref{hereditary}, it is sufficient to prove that every blowup of every prime ($P_5$, house, bull)-free graph is recolorable. Let $G$ be a prime ($P_5$, house, bull)-free graph. By Lemma \ref{join}, we may assume that $G$ does not have a universal vertex. By Lemma \ref{5vertex}, we may assume that $G$ has at least 6 vertices. By Theorem \ref{strPHB}, $G$ is a bipartite graph or a co-bipartite graph. Then the result follows from Lemma \ref{blowupbip} and Lemma \ref{blowupco-bip}.
\end{proof}

 A \emph{thin spider} is a graph whose vertex set can be decomposed into a clique $K$ and an independent set $S$, such that $|K|$ = $|S|$ or $|K|$ = $|S|$+1, and the edges between $S$ and $K$ form a matching, leaving at most one vertex of $K$ with no neighbor in $S$.

\begin{lemma}\label{blowupspider}
    Every blowup of a thin spider or the complement of a thin spider is recolorable.
\end{lemma}
\begin{proof}
    Let $H$ be a blowup of a thin spider or the complement of a thin spider. Since every thin spider and every complement of a thin spider is chordal, the graph $H$ is recolorable by Lemma \ref{bchordal}.
\end{proof}

\begin{theorem}[\cite{Fouquet2}]\label{semiparse}
    Let $G$ be a prime semi-$P_4$-sparse graph. Then either $G$ or its complement is a bipartite graph or a thin spider.
\end{theorem}

We now prove Theorem \ref{PHK}.

\begin{proof}[Proof of Theorem \ref{PHK}]
     By Theorem \ref{hereditary}, it is sufficient to prove that every blowup of every prime semi-$P_4$-sparse graph is recolorable. Let $G$ be a prime semi-$P_4$-sparse graph. Note that $G$ is $P_5$-free. By Theorem \ref{semiparse}, $G$ or its complement is a bipartite graph or a thin spider. If $G$ or its complement is a bipartite graph, then the result follows from Lemma \ref{blowupbip} or Lemma \ref{blowupco-bip}, respectively. If $G$ or its complement is a thin spider, then the result follows from Lemma \ref{blowupspider}. 
\end{proof}

\begin{theorem}[\cite{Fouquet1}]\label{p5house}
    Every prime ($P_5$, house)-free graph is either a $C_5$ or $C_5$-free.
\end{theorem}

We now prove Theorem \ref{PHC}.

\begin{proof}[Proof of Theorem \ref{PHC}]
    It is obvious that if every ($P_5$, house)-free graph is recolorable, then every ($P_5$, house, $C_5$)-free graph is recolorable. Now assume that every ($P_5$, house, $C_5$)-free graph is recolorable. By Theorem \ref{hereditary}, it is sufficient to prove that every blowup of every prime ($P_5$, house)-free graph is recolorable. Let $G$ be a prime ($P_5$, house)-free graph. If $G$ is $C_5$-free, then every blowup of $G$ is ($P_5$, house, $C_5$)-free and hence recolorable. If $G$ contains a $C_5$, then by Theorem \ref{p5house}, $G$ is isomorphic to $C_5$. Since $C_5$ is 3$K_1$-free, every blowup of $C_5$ is recolorable by Lemma \ref{blow3K1-free}.
\end{proof}

\section{Conclusion}\label{conclusion}

Let $\mathcal{G}$ be a hereditary class of graphs. In Theorem \ref{hereditary}, we proved that every graph in $\mathcal{G}$ is recolorable if every blowup of every prime graph in $\mathcal{G}$ is recolorable. We have two questions.\\

\begin{figure}[h]
\centering
\begin{tikzpicture}[scale=0.5]
\tikzstyle{vertex}=[circle, draw, fill=black, inner sep=0pt, minimum size=5pt]

    \node[vertex](1) at (4,0) {};
    \node[vertex](2) at (4*cos{60},4*sin{60}){};
    \node[vertex](3) at (4*cos{120},4*sin{120}){};
    \node[vertex](4) at (4*cos{180},4*sin{180}){};
    \node[vertex](5) at (4*cos{240},4*sin{240}){};
    \node[vertex](6) at (4*cos{300},4*sin{300}){};
    \node[vertex](7) at (4*cos{60}-2,4*sin{60}+2){};
     
    \draw(1)--(6); \draw(1)--(2); \draw(2)--(3); \draw(3)--(4); \draw(5)--(4); \draw(5)--(6); \draw(2)--(7); \draw(3)--(7);

    \node[vertex, label=right:1](11) at (17,0) {};
    \node[vertex, label=right:3](12) at (5*cos{60}+12,5*sin{60}){};
    \node[vertex, label=left:5](13) at (5*cos{120}+12,5*sin{120}){};
    \node[vertex, label=left:1](14) at (5*cos{180}+12,5*sin{180}){};
    \node[vertex, label=left:3](15) at (5*cos{240}+12,5*sin{240}){};
    \node[vertex, label=right:5](16) at (5*cos{300}+12,5*sin{300}){};
    \node[vertex, label=right:1](17) at (5*cos{60}+9.5,5*sin{60}+2){};
    \node[vertex, label=left:2](21) at (15,0) {};
    \node[vertex, label=below:4](22) at (3*cos{60}+12,3*sin{60}){};
    \node[vertex, label=below:6](23) at (3*cos{120}+12,3*sin{120}){};
    \node[vertex, label=right:2](24) at (3*cos{180}+12,3*sin{180}){};
    \node[vertex, label=above:4](25) at (3*cos{240}+12,3*sin{240}){};
    \node[vertex, label=above:6](26) at (3*cos{300}+12,3*sin{300}){};

    \draw(11)--(16); \draw(11)--(12); \draw(12)--(13); \draw(13)--(14); \draw(15)--(14); \draw(15)--(16); \draw(12)--(17); \draw(13)--(17); \draw(21)--(26); \draw(21)--(22); \draw(22)--(23); \draw(23)--(24); \draw(25)--(24); \draw(25)--(26); \draw(21)--(11); \draw(21)--(12); \draw(21)--(16); \draw(22)--(12); \draw(22)--(11); \draw(22)--(13); \draw(23)--(13); \draw(23)--(12); \draw(23)--(14); \draw(24)--(13); \draw(24)--(14); \draw(24)--(15); \draw(25)--(14); \draw(25)--(15); \draw(25)--(16); \draw(26)--(15); \draw(26)--(16); \draw(26)--(11); \draw(22)--(17); \draw(23)--(17);    
\end{tikzpicture}
\caption{A graph $G$ (left) and a blowup $G^{'}$ of it (right).}
\label{fig:1}
\end{figure}
\textbf{\textit{Question 1}:} If every prime graph in $\mathcal{G}$ is recolorable, then is every graph in $\mathcal{G}$ recolorable? \\
We answered Question 1 positively for 2$K_2$-free graphs and diamond-free graphs, and would like to know if it is true in general.

In light of Theorem \ref{hereditary}, to answer Question 1 positively, one might check if every blowup of a prime recolorable graph is recolorable. However, this is not true.
% Furthermore, we can ask if every blowup of a prime recolorable graph recolorable. We answer this in the negative.

It is easy to see that the graph $G$ in Figure \ref{fig:1} (left) is a prime 3-chromatic graph and it is recolorable, but the blowup $G^{'}$ of $G$ (right) is 5-chromatic and $R_6(G^{'})$ is disconnected and hence $G^{'}$ is not recolorable. However, we can see that $G$ has a prime induced subgraph, $C_6$, which is not recolorable.
\par
We conjecture the following:
\begin{conjecture}
Let $G$ be a prime graph. Every blowup of $G$ is recolorable if and only if every induced subgraph of $G$ is recolorable.
\end{conjecture}

\section*{Acknowledgement}
We thank Ch\'inh~T.~Ho\`ang for helpful suggestions. Some of the work in this paper appeared in \cite{Belavadi2024}.

\end{document}